\documentclass[10pt]{amsart}
\usepackage{amsmath,amssymb,enumerate}

\newtheorem{thm}{Theorem}[section]
 \numberwithin{equation}{section} %% Comment out for sequentially-numbered
 \numberwithin{figure}{section} %% Comment out for sequentially-numbered
 \theoremstyle{plain}
 \theoremstyle{plain}    
 \newtheorem{cor}[thm]{Corollary} %%Delete [thm] to re-start numbering
 \theoremstyle{plain}    
 \newtheorem{prop}[thm]{Proposition} %%Delete [thm] to re-start numbering
 \theoremstyle{plain}    
 \newtheorem{lem}[thm]{Lemma} %%Delete [thm] to re-start numbering
 
  \newtheorem{defi}[thm]{Definition}
\newtheorem{exa}[thm]{Example}

\newcommand{\N}{\mathbb{N}}
\newcommand{\R}{\mathbb{R}}
\newcommand{\C}{\mathbb{C}}

\newcommand{\Pc}{\mathcal{P}}
\newcommand{\Cc}{\mathcal{C}}
\newcommand{\Uc}{\mathcal{U}}
\newcommand{\Ccu}{\mathcal{C}^{\uparrow}}
\newcommand{\f}{\varphi}
\newcommand{\p}{\psi}
\newcommand{\Ec}{\mathcal{E}}
\newcommand{\Fc}{\mathcal{F}}
\newcommand{\Mcc}{\mathcal{M}}
\newcommand{\Fcm}{\mathcal{F}_{\mu}}

\newcommand{\SH}{\mathcal{SH}}
\newcommand{\ind}{{\bf 1}}
\newcommand{\weak}{\rightharpoonup}

\newcommand{\setdef}{\ \big\vert \ }
\newcommand{\Capa}{{\rm Cap}}

\newcommand{\vol}{{\rm Vol}}

\newcommand{\SHXo}{\SH_m(X,\omega)}
\newcommand{\vep}{\varepsilon}
\newcommand{\PmXo}{\mathcal{P}_m(X,\omega)}
\newcommand{\EcXo}{\mathcal{E}(X,\omega,m)}
\newcommand{\Ect}{\mathcal{E}^1(X,\omega,m)}
\newcommand{\EctC}{\mathcal{E}^1_C(X,\omega,m)}
\newcommand{\Capm}{\Capa_{\omega,m}}

\newtheorem{theorem}{Theorem}

\begin{document}
\title[Degenerate Complex Hessian equations]{Degenerate complex Hessian equations on compact K\"ahler manifolds} 
\date{\today \\
The first-named author is  partially supported by the french ANR project MACK} 

\author{Chinh H. Lu}
\address{Chalmers University of Technology \\ Mathematical Sciences\\
412 96 Gothenburg\\ Sweden}

\email{chinh@chalmers.se}

\author{Van-Dong Nguyen} 

\address{Department of Mathematics-Informatics \\
Ho Chi Minh city University of Pedagogy\\ 280 An Duong Vuong\\ Ho Chi Minh city, Vietnam}

\email{dong5591@gmail.com}

\begin{abstract}
Let $(X,\omega)$ be a compact K\" ahler manifold of dimension $n$ and fix $m\in \N$ such that $1\leq m \leq n$. We prove that  any 
$(\omega,m)$-sh function can be approximated from above by smooth 
$(\omega,m)$-sh functions. A  potential theory for the complex Hessian equation is also developed  which generalizes the classical  pluripotential theory on compact K\"ahler manifolds. We then use novel variational tools due to Berman, Boucksom, Guedj and Zeriahi  to study degenerate complex Hessian equations.
\end{abstract}

\maketitle
\tableofcontents

\section{Introduction}
Let $(X,\omega)$ be a compact  K\"{a}hler manifold of complex dimension $n$. Let $m$ be a natural number   between $1$ and  $n$. Denote by $d,d^c$ 
the usual real differential operators $d:=\partial+\bar{\partial}$, 
$d^c=\frac{\sqrt{-1}}{2\pi}(\bar{\partial}-\partial)$ 
so that $dd^c =\frac{i}{\pi}\partial\bar{\partial}$.

The complex $m$-Hessian equation can be considered as an interpolation between the classical Poisson equation (corresponds to the case when $m=1$) and the
complex Monge-Amp\`ere equation (corresponds to $m=n$) which has been studied intensively in the recent years with many applications to complex geometry. For recent developments of the latter, see \cite{Y78,Bl05,GZ05,GZ07,Kol98,Kol05} and the references therein.

The study of complex Hessian equations was initiated by Li in \cite{Li04} where 
he solved the Dirichlet problem with smooth data on a smooth strictly 
$m$-pseudoconvex domain in $\C^n$. B{\l}ocki \cite{Bl05} developed the first steps of a potential theory for this equation and suggested a study of the corresponding equation on compact K\"ahler manifolds which is analogous to the complex Monge-Amp\`ere equation. 

The non-degenerate complex Hessian equation  is of the following form
\begin{equation}\label{eq: Hes intro}
(\omega +dd^c \f)^m\wedge \omega^{n-m} =f \omega^n,
\end{equation}
where $0 < f\in \Cc^{\infty}(X)$ satisfies the necessary condition $\int_X f\omega^n= \int_X \omega^n$. It  has been studied by Kokarev \cite{Kok10}, Jbilou \cite{Jb10} and
Hou-Ma-Wu \cite{H09},\cite{HMW10}. In \cite{Jb10} and \cite{H09} the authors  independently proved that equation (\ref{eq: Hes intro}) has a unique (up to an additive constant)
smooth \textit{admissible} solution provided the metric $\omega$ has non-negative holomorphic bisectional curvature. This technical assumption turned out to be very strong since manifolds carrying such metrics are very restrictive thanks to a uniformization theorem of Mok (see \cite{Mok88}). Another effort from Hou-Ma-Wu \cite{HMW10} provided an a priori \textit{almost} $\Cc^2$-estimate without any curvature assumption. The authors also mentioned that their estimate 
can also be used in a blow-up analysis which actually reduced the problem of solving the complex Hessian equation on a compact K\"ahler manifold to proving 
a Liouville-type theorem for $m$-subharmonic functions in $\C^n$. The latter has been recently solved by Dinew and Ko{\l}odziej \cite{DK12} which confirmed the smooth resolution of equation (\ref{eq: Hes intro}) in full generality. Dinew and Ko{\l}odziej also gave a very powerful $\Cc^0$-estimate in \cite{DK11} which allows one to find continuous weak solution of the degenerate complex Hessian equation with $f\in L^p(X)$ for some $p>n/m$. 

The \textit{real} Hessian equation has been studied intensively with many geometric applications. For a survey of this theory we refer the reader to \cite{Tr95}, \cite{TW99}, \cite{W09}, \cite{CW01} and the references therein. 
Some similar non-linear elliptic equations of have appeared in the study of geometric deformation flows such as the $J$-flow (see \cite{SW}). From the point of view of non-linear elliptic partial differential equations the complex Hessian equation is an interesting and important object. Recently another general (and powerful) $\Cc^{2,\alpha}$ estimate for equations of this type has been obtained in \cite{TWWY}. 
In \cite{AV10} equations of complex Hessian type appear in the study of quaternionic geometry. 
Thus it is expected that the complex Hessian equation we considered here will have some interesting geometric applications.

The notion of $(\omega,m)$-subharmonic functions has been introduced in \cite{DK11} and studied by the first-named author in a systematic way in \cite{Lu13b}. It was not  clear how to define the complex Hessian operator 
for any bounded $(\omega,m)$-sh function due to a lack of regularization process. Recently Plis \cite{Pl13} proved that  one can approximate continuous strictly $(\omega,m)$-sh functions by smooth ones. 
In this paper we show that one can globally approximate any $(\omega,m)$-subharmonic functions from above by a sequence of smooth
 $(\omega,m)$-subharmonic functions. 
In particular, the class
$\PmXo$ introduced in \cite{Lu13a} consists of all $(\omega,m)$-subharmonic functions.

For any  upper semicontinuous function $f$ we define the projection of $f$ on the space of $(\omega,m)$-subharmonic functions by:
$$
P(f) := \sup \left\{u\in \SHXo \setdef u\leq f\right\}.
$$ 

\medbreak

Using Berman's technique \cite{Berman13} combined with the viscosity method  by Eyssidieux, Guedj and Zeriahi \cite{EGZ11} we can prove that the projection of a smooth function is continuous.  Moreover, we can also prove the orthogonal relation without 
solving the local Dirichlet problem. Let us stress that our method is new even in the case of complex Monge-Amp\`ere equations ($n=m$).

\begin{theorem}
\label{theorem: orthogonal}
{\it Let $(X,\omega)$ be a compact 
K\"ahler manifold of dimension $n$ and fix $m\in \N$ such that
$1\leq m\leq n$. Let $h$ be a continuous function on $X$. The 
$(\omega,m)$-subharmonic function $P(h)$ is continuous and its Hessian measure is a non-negative measure supported on the contact set $\{P(h) = h\}$. }
\end{theorem}

Following \cite{EGZ13} we can  approximate any $(\omega,m)$-sh function from above by a sequence of smooth $(\omega,m)$-sh functions.

\begin{theorem}
\label{theorem: smooth approx}
{\it For every $u\in \SHXo$ we can find a decreasing sequence of smooth 
$(\omega,m)$-subharmonic functions on $X$ which decreases to $u$ on
$X$.}
\end{theorem}

The approximation theorem (Theorem \ref{theorem: smooth approx}) was known to hold for \textit{continuous} $(\omega,m)$-sh functions by a recent paper of Plis (\cite{Pl13}) while the same question for any $(\omega,m)$-sh function is still open until now. One can easily figure out that
we only need to regularize any  $(\omega,m)$-sh function by continuous functions and then apply Plis's result.  Let us emphasize that  we  prove the approximation theorem independently by combining the "$\beta$-convergence" method of 
Berman \cite{Berman13} and the envelope method of Eyssidieux-Guedj-Zeriahi \cite{EGZ13}. 
\medbreak

Being able to regularize bounded $(\omega,m)$-subharmonic functions we can define the complex Hessian operator for these functions following the pluripotential method of Bedford and Taylor \cite{BT76}. We then can adapt many arguments in pluripotential theory for complex Monge-Amp\`ere equations to a potential theory
for complex Hessian equations on compact K\"ahler manifolds. We can also mimic the definition of the class $\Ec(X,\omega)$ for $\omega$-psh functions introduced in \cite{GZ07} to define a similar class of $(\omega,m)$-subharmonic functions $\EcXo$. Using the variational approach inspired 
by Berman, Boucksom, Guedj and Zeriahi \cite{BBGZ13} we can solve very degenerate complex Hessian equations with  right-hand sides being positive measures vanishing on $m$-polar sets. The two major steps to apply this method are the regularization process and the orthogonal relation which have been established in the previous results. We also give  simpler proof of the differentiability of the energy functional composed with the projection (see Lemma \ref{lem: BB08}).

\medskip

\begin{theorem}
\label{thm: intro var}
{\it
Let $\mu$ be a positive Radon measure on $X$ satisfying the compatibility condition $\mu(X)= \int_X \omega^n$. Assume that 
$\mu$ does not charge $m$-polar subsets of $X$. Then there exists a  solution $\f\in \EcXo$ to 
$$
(\omega+dd^c \f)^m \wedge \omega^{n-m} =\mu. 
$$
}
\end{theorem}

\medbreak

\noindent{\bf Acknowledgements.} We thank Robert Berman for many useful discussions. We are indebted to Eleonora Di Nezza and Ahmed Zeriahi for a very careful reading of a previous draft version of this paper.

\section{Preliminaries}
In this section we recall basic facts about $(\omega,m)$-subharmonic functions. We refer the readers to \cite{Bl05}, \cite{DK11}, \cite{DK12}, \cite{Lu13a}, \cite{Lu13b}, \cite{Cuong1}, \cite{Cuong2}, 
\cite{SA12}, \cite{Moh14} for more details. We always denote by $(X,\omega)$  a compact K\"{a}hler manifold. By $(M,\omega)$ we denote a 
K\"ahler manifold which is not necessary compact. Let $n$ be the complex dimension of the manifold  and fix an integer $m$ such that $1\leq m\leq n$. 
We denote $\Cc^{\uparrow}$ the space of upper semicontinuous functions.

\subsection{$m$-Subharmonic functions}
\begin{defi}

\label{def: m-sh smooth}
A real $(1,1)$-form $\alpha$ is called $m$-positive on $M$ if 
the following inequalities hold in the classical sense:
$$
\alpha^k \wedge \omega^{n-k} \geq 0, \ \forall k=1,..., m.
$$
A function $\f \in \Cc^2(M)$ is called $m$-subharmonic ($m$-sh for short) on 
$M$ if the $(1,1)$-form $dd^c \f$ is $m$-positive on $M$. \\
A current  $T$ of bidegree $(n-p,n-p)$, $p\leq m$, is called $m$-positive on
$M$ if for any $m$-positive
$(1,1)$-forms  $\alpha_1,...,\alpha_{p}$ the following inequality holds in the  sense of currents :
$$
\alpha_1\wedge \cdots \wedge \alpha_{p} \wedge T \geq 0 .
$$
\end{defi}
For each $k\geq 1$ the symmetric polynomial of degree $k$ on $\R^n$ is defined by 
$$
S_k(\lambda) : = \sum_{1\leq i_1<\cdots <i_k\leq n} \lambda_{i_1}\cdots \lambda_{i_k} ,\ \ \ \lambda:= (\lambda_1,\cdots,\lambda_n)\in \R^n .
$$
Let $\f\in \Cc^2(M)$ and set $\lambda:=\lambda_{\f}(x)\in \R^n$ the vector of eigenvalues of $dd^c \f$ at $x$ with respect to $\omega$. Then $\f$ is $m$-subharmonic in $\Omega$ if and only if 
$$
S_k(\lambda_{\f}(x)) \geq 0, \ \forall x\in M, \ \forall k = 1, ... , m. 
$$
The following lemma follows immediately from G{\aa}rding's inequality \cite{Ga59} (see also \cite{Bl05}).
\begin{lem}

\label{lem: Garding}
Let $u\in \Cc^2(M)$. Then $u$ is $m$-subharmonic in $M$
if and only if for every $m$-positive $(1,1)$-forms 
$(\alpha_1,...,\alpha_{m-1})$ on $M$ the following inequality holds in the classical sense:
$$
dd^c u  \wedge  \alpha_1 \wedge \cdots  \wedge \alpha_{m-1}
\wedge \omega^{n-m} \geq 0 .  
$$ 
\end{lem}
This lemma suggests a way to extend Definition \ref{def: m-sh smooth} for non-smooth functions. 

\begin{defi}
Assume that $u\in \Ccu(M)$ is locally integrable on $M$. Then $u$ is called $m$-sh on $M$ if 

(i) for any $m$-positive $(1,1)$-forms 
$\alpha_1,...,\alpha_{m-1}$ on $M$, 
the following inequality holds in the weak sense of currents
on $M$ :
$$
dd^c u \wedge  \alpha_1 \wedge \cdots  \wedge \alpha_{m-1}
\wedge \omega^{n-m} \geq 0 .
$$

(ii) if $v\in \Ccu(M)$ is locally integrable, satisfies (i) and 
$u=v$ almost everywhere on $M$ then $u\leq v$.
\end{defi}
We denote by $\SH_m(M)$ the class of all $m$-sh functions on $M$. If 
$M$ is compact this class contains only constant functions. We will study instead the class of $(\omega,m)$-subharmonic functions (see the next section).

\begin{defi}
A function $u$ is called strictly $m$-sh on $M$ if for every 
function $\chi\in \Cc^2(M)$ there exists $\vep>0$ such that
$u+\vep \chi$ is $m$-sh on $M$.
\end{defi}

When $M=\Omega$ a bounded $m$-hyperconvex domain in $\C^n$, we recover the definition of $m$-subharmonic functions introduced in \cite{Bl05},
\cite{SA12}, \cite{Cuong1}, \cite{Cuong2},\cite{Lu13c}.

\subsection{$(\omega,m)$-subharmonic functions}
\begin{defi}
Let $u\in \Ccu(X)$ be an integrable function. Then $u$ is called 
$(\omega,m)$-subharmonic on $X$ if 
 for every local chart $\Omega$ the function $u+\rho$ is $m$-sh in $\Omega$, where $\rho$ is a local potential of $\omega$. Here, we regard $(\Omega,\omega\vert_{\Omega})$
as an open K\"ahler manifold. The notion of $m$-sh functions on $\Omega$ is defined in the previous subsection.  
\end{defi}

When $m=1$, $(\omega,1)$-sh functions on $X$ are just $\omega$-subharmonic functions on $X$. When $m=n$, $(\omega,m)$-sh functions are exactly $\omega$-plurisubharmonic functions which have been studied intensively by many authors in the recent years.

It follows from Lemma \ref{lem: Garding} that a function $u\in \Cc^2(X)$ is $(\omega,m)$-sh if and only if the associated $(1,1)$-form $\omega + dd^c \f$ is $m$-positive on $X$. In general $u$ is 
$(\omega,m)$-sh if the current $\omega + dd^c \f$ is $m$-positive on 
$X$. 

\begin{defi}
A function $u$ is called strictly $(\omega,m)$-sh on $X$ if
for every function $\chi\in \Cc^2(X)$ there exists $\vep >0$ such that 
$u+\vep \chi $ is $(\omega,m)$-sh on $X$.
\end{defi}

Continuous strictly $(\omega,m)$-sh functions on $X$ can also be approximated from above by smooth ones. This is the content of the next theorem due to Plis \cite{Pl13}:

\begin{thm}\cite{Pl13}
\label{thm: Plis reg compact}
Let $(X,\omega)$ be a compact K\"ahler manifold and $u$ be a continuous strictly $(\omega,m)$-subharmonic function on
$X$. Let $h$ be a continuous  function on $X$ such that $h>0$. Then there exists a smooth strictly $(\omega,m)$-sh function  $\f$ on $X$
such that 
$$
u \leq \f \leq  u + h .
$$ 
\end{thm}

\subsection{The complex Hessian operator}

\label{sect: Hessian operator}
We briefly recall basic facts about the class $\PmXo$ and the complex Hessian operator introduced in \cite{Lu13a}. 

Let $U\subset X$ be an open subset. The Hessian $m$-capacity of $U$ is defined by
$$
\Capm(U) := \sup \left\{ \int_U H_m(u) \setdef u\in \SHXo \cap \Cc^2(X), \ -1\leq u\leq 0  \right\},
$$
where for a smooth function $u$, we denote $H_m(u):=(\omega+dd^c u)\wedge \omega^{n-m}$.
\begin{defi}
Let $\f\in \SHXo$. By definition $\f$ belongs to $\PmXo$ if there exists a sequence $(\f_j) \subset \SHXo \cap \Cc^2(X)$ which converges 
to $\f$ quasi-uniformly on $X$, i.e. for any $\vep>0$ there exists 
an open subset $U$ such that $\Capm(U)<\vep$ and $\f_j$ converges uniformly to $\f$ on $X\setminus U$. 
\end{defi}
We will show in the next section that  $\PmXo = \SHXo$. It follows from the definition that every $\f\in \PmXo$ is quasi-continuous, i.e.
for any $\vep>0$ we can find an open subset $U$ such that 
$\Capm(U)<\vep$ and the restriction of $\f$ on $X\setminus U$
is continuous.

Assume that $\f\in \PmXo$ is bounded. Let $(\f_j)$ be a sequence of 
functions in $\SHXo\cap \Cc^2(X)$ which converges quasi-everywhere on 
$X$ to $\f$. Then the sequence  $H_m(\f_j)$
converges weakly 
to some positive Radon measure $\mu$. This measure $\mu$ does not depend on the choice of the sequence $(\f_j)$ and is defined to be the
Hessian measure of $\f$:
$$
(\omega+dd^c \f_j)^m\wedge \omega^{n-m} =: H_m(\f_j) \weak H_m(\f).
$$
The class $\PmXo$ is stable under taking maximum and under decreasing sequence. 

\subsection{Viscosity vs potential sub-solution}
Let $0\leq F$ be a continuous function on $X$ and $u$ be an upper semicontinuous function on $X$. Let $x_0\in X$ and $q$ be a $\Cc^2$ function in a small neighborhood $V$ of $x_0$. We say that $q$ touches 
$u$ from above (in $V$) at $x_0$ if $q\geq u$ in $V$ with equality at $x_0$.

We say that $u$ is a viscosity sub-solution of equation
\begin{equation}
\label{eq: vis sub}
F\omega^n -(\omega +dd^c \f)^m \wedge \omega^{n-m} =0
\end{equation}
if for any $x_0\in X$ and any $\Cc^2$ function $q$ in a neighborhood of $x_0$ which touches $u$ from above at $x_0$ then the following inequality holds at $x_0$
\begin{equation*}
F\omega^n -(\omega +dd^c q)^m \wedge \omega^{n-m} \leq 0.
\end{equation*}

The following result has been proved by Plis (\cite{Pl13}) using \cite{HL13}. This will play an important role in our regularization theorem.
\begin{lem}\cite{Pl13}
\label{lem: vis pot 0}
Let $u$ be a $(\omega,m)$-subharmonic function on $X$. Then $u$ is a 
viscosity sub-solution of (\ref{eq: vis sub}) with $F\equiv 0$. More precisely, for any $x_0\in X$ and any $\Cc^2$ function $q$ in a neighborhood of $x_0$ which touches $u$ from above at $x_0$ then the following inequalities hold at $x_0$:
$$
(\omega +dd^c q)^k\wedge \omega^{n-k} \geq 0 , \ \forall k=1,...,m.
$$
\end{lem}
We also need a generalized version of the above result. The idea is to adapt some useful tricks in \cite{HL13}.  
\begin{lem}
\label{lem: vis pot gen}
Let $F$ be a non-negative continuous function on $X$. Let $u\in \PmXo \cap \Cc(X)$ be a potential solution of equation (\ref{eq: vis sub}). Then $u$ is also a viscosity sub-solution of (\ref{eq: vis sub}). \end{lem}
\begin{proof}
We argue by contradiction. Assume that there exists $x_0\in X$,  $B=\bar{B}(x_0,r)$ a small open ball centered at $x_0$ and $q\in \Cc^2(B)$ such that 
$q$ touches $u$ from above in $B$ at $x_0$ but 
\begin{equation}
\label{eq: vis pot sub 1}
F\omega^n -(\omega +dd^c q)^m \wedge \omega^{n-m} > \vep ,
\end{equation}
at $x_0$ for some positive constant $\vep>0$. Since $F$ is continuous, by shrinking $B$ a little bit we can assume that (\ref{eq: vis pot sub 1}) holds for every point in $B$. 

Fix $\delta>0$. It follows from 
\cite{DK11}, \cite{DK12}  that we can find $u_{\delta}\in \SHXo\cap \Cc^{\infty}(X)$ such that
$$
\sup_X |u_{\delta}-u| < \delta r^2/2 \ , \sup_X |F_{\delta}-F|<\delta/2\ ,  \ {\rm and}\  (\omega +dd^c u_{\delta})^m \wedge \omega^{n-m} = 
F_{\delta} \omega^n.
$$
Consider the function 
$$
\phi_{\delta}(x) := u_{\delta} (x) -q(x) -\delta |x-x_0|^2 , \ \  x\in B=\bar{B}(x_0,r).
$$
Let $x_{\delta}\in B$ be a maximum point of $\phi_{\delta}$ in $B$. Observe that if $x\in \partial B$ we have
$$
\phi_{\delta}(x) < u(x) + \delta r^2/2 -q(x) -\delta r^2 \leq -\delta r^2/2,
$$
while $\phi_{\delta}(x_0) > u(x_0) -\delta r^2/2 -q(x_0) = -\delta r^2/2$. Thus $x_{\delta}$ is in the interior of $B$ and hence the maximum principle yields that 
$$
(\omega +dd^c (q + \delta |x-x_0|^2))^m \wedge \omega^{n-m} \geq F_{\delta}\omega^n
$$
holds at $x_{\delta}$. Letting $\delta\downarrow 0$ we can find $\bar{x}\in B$ such that the following holds at $\bar{x}$
$$
(\omega +dd^c q )^m \wedge \omega^{n-m} \geq F\omega^n.
$$
This yields a contradiction since (\ref{eq: vis pot sub 1}) holds in $B$. 

\end{proof}

\section{Approximation of  $(\omega,m)$-subharmonic functions}
In this section we prove the approximation theorem. Recall that it follows from the recent work of Plis \cite{Pl13} (Theorem \ref{thm: Plis reg compact})
that any continuous $(\omega,m)$-sh function can be uniformly approximated by smooth ones. Thus one needs only to approximate any $(\omega,m)$-sh function by continuous ones. Let us stress that  our proof is independent of  Plis's result. We immediately regularize 
$(\omega,m)$-sh functions by using recent methods of Berman \cite{Berman13} and Eyssidieux-Guedj-Zeriahi \cite{EGZ13}.

We first define the projection of any function on the class of  $(\omega,m)$-sh functions. Let $f$ be any upper semicontinuous function such that there is a $(\omega,m)$-sh function lying below $f$. We define
$$
P(f):= \sup\left\{ v\in \SHXo \setdef v\leq f \ {\rm on }\ X 
\right\}.
$$
It is clear that $P(f)^*$ is again a candidate defining 
$P(f)$. Then $P(f)^*\leq P(f)$ which implies that $P(f)=P(f)^*$
is a $(\omega,m)$-sh function. This is the maximal $(\omega,m)$-sh function lying below $f$. The following observation follows straightforward from the definition:

\begin{lem}\label{lem: proj stable}
Let $f, g$ be two bounded upper semicontinuous functions on $X$. Then 
$$
\sup_X |P(f)-P(g)| \leq \sup_X |f-g| .
$$
\end{lem}

Let $f$ be any function of class $\Cc^2$ on $X$. We define $H_m(f)^+$ 
to be the non-negative part of $H_m(f)$, i.e.
$$
H_m(f)^+(z) := \max \left[(\omega+dd^cf)^m\wedge \omega^{n-m} (z) , 0\right] .
$$
Observe that for any $f\in \Cc^2(X)$, the non-negative part of the Hessian measure of $f$ is a non-negative measure having  continuous density. This measure also has positive mass. It follows from \cite{Lu13a} that for every 
$\beta>0$ the following equation has a unique continuous solution in the class $\PmXo$:
\begin{equation}
\label{eq: beta 1}
(\omega +dd^c \f)^m\wedge \omega^{n-m} = e^{\beta(\f-f)} 
\left(H_m(f)^+  + \frac{\omega^n}{\beta} \right) .
\end{equation}

\begin{thm}
\label{thm: upper bound}
Let $f\in \Cc^2(X)$. For each $\beta>1$, let $u_{\beta}\in \PmXo$ be the unique solution to equation (\ref{eq: beta 1}). Then $u_{\beta}\leq f$ on $X$. Moreover, when $\beta$ goes to $+\infty$, $u_{\beta}$ converges uniformly on $X$ to the upper envelope 
$$
P(f):= \sup\left\{ v\in \SHXo \setdef v\leq f \ {\rm on }\ X 
\right\} .
$$ 
In particular $P(f)$ belongs to $\PmXo \cap \Cc(X)$ and satisfies 
$$
H_m(P(f)) \leq  \ind_{\{P(f)=f\}} H_m(f).
$$ 
\end{thm}
This Theorem is the principal result of our paper. Right after we know how to regularize singular functions the other parts of the potential theory can be easily adapted from the classical pluripotential theory for Monge-Amp\`ere equations.
Let us also stress that our proof is new even in the case of complex Monge-Amp\`ere equations.

\begin{proof}
Fix $\beta>1$. To simplify the notation we set $\f:=u_{\beta}$. 
It follows from Lemma \ref{lem: vis pot gen} that $\f$ is also a viscosity sub-solution of equation
$$
e^{\beta(\f-f)} \left(H_m(f)^+  + \frac{\omega^n}{\beta} \right) -(\omega +dd^c \f)^m\wedge \omega^{n-m}=0. 
$$
Let $x_0$ be a point where $\f-f$ attains its maximum on $X$. Then $f-f(x_0)+\f(x_0)$ touches $\f$ from above at $x_0$. By definition of viscosity sub-solutions we get
$$
e^{\beta(\f(x_0)-f(x_0))} \left(H_m(f)^+  + \frac{\omega^n}{\beta} \right) (x_0) -(\omega +dd^c f)^m\wedge \omega^{n-m}(x_0) \leq 0.
$$
We then infer that $\f(x_0)\leq f(x_0)$ which proves that $\f-f\leq 0$
on $X$. 

Now, fix $\beta>\gamma>1$. Since $u_{\beta}\leq f$, it is easy to see that 
$$
(\omega+dd^c u_{\beta})^m\wedge \omega^{n-m} \leq 
e^{\gamma(u_{\beta}-f)}\left(H_m(f)^+  + \frac{\omega^n}{\gamma} \right) .
$$
It thus follows from the comparison principle (see \cite[Corollary 3.15]{Lu13a}) that $u_{\beta} \geq u_{\gamma}$. Therefore the sequence
$(u_{\beta})$ 
converges. Observe also that the right-hand side of (\ref{eq: beta 1}) has uniformly bounded density. It then follows from \cite{DK11} and \cite{Lu13a} that $(u_{\beta})$ converges uniformly on $X$ to $u\in \PmXo \cap \Cc(X)$ which satisfies
$$
(\omega +dd^c u)^m\wedge \omega^{n-m} \leq \ind_{\{u=f\}}  H_m(f)^+.
$$
\medbreak

Now, we prove that $u=P(f)$. Let us fix $h\in \SHXo$ such that 
$h\leq f$. We need to show that $h\leq u$. The idea behind the proof is quite simple: since $H_m(u)$ vanishes on $\{u<f\}$, $u$ must be maximal there, hence $u$ dominates any candidate defining $P(f)$.

Fix $\vep>0$ and set 
$U:= \{u < f-\vep\}$. Write 
$$
(\omega +dd^c u_{\beta})^m\wedge \omega^{n-m} = f_{\beta} \omega^n,
$$
where $f_{\beta}$ is a non-negative continuous function on $X$. Since 
$u_{\beta}$ converges uniformly on $X$ to $u$ and $f_{\beta}$ converges uniformly to $0$ on $U$, we can find $\beta>0$ very big such that 
$$
\sup_U f_{\beta} < \vep^m/2 \ \ {\rm and}\ \ \sup_X |u_{\beta}-u|<\vep/2 .
$$
Now, since $f_{\beta}$ is continuous on $X$ there is a sequence of smooth positive functions $g_{\beta}^j$ converging uniformly to $f_{\beta}$ on $X$ such that $\int_X g_{\beta}^j \omega^n =\int_X \omega^n$.  Let $v_{\beta}^j$ be the corresponding smooth solutions to the complex Hessian equations $H_m(v_{\beta}^j)=g_{\beta}^j\omega^n$.
Then it follows from \cite{DK11} that $v_{\beta}^j$ converges uniformly to $u_{\beta}$ on $X$. Now for $j$ large enough we have found $v_{\beta}$ (we drop the index $j$) a smooth $(\omega,m)$-sh functions on $X$ such that
$$
H_m(v_{\beta})= g_{\beta}\omega^n,\ \ \sup_X |v_{\beta}-u_{\beta}| < \vep/2 \ \ {\rm and} \ \ \sup_X |g_{\beta}-f_{\beta}| < \vep^m/2 .
$$
In particular, we have 
$$
H_m(v_{\beta})= g_{\beta}\omega^n,\ \ \sup_X |v_{\beta}-u| < \vep \ \ {\rm and} \ \ \sup_U g_{\beta} < \vep^m .
$$
Consider the function 
$$
\phi:= h - (1+\delta) v_{\beta}, 
$$
where $\delta=\vep/(1-\vep)$.
Since $\phi$ is upper semicontinuous on $X$ compact, it attains its maximum on $X$ at some $y_0\in X$. 
 
Assume that $y_0\in U$.  Then the function $(1+\delta)v_{\beta} -(1+\delta) v_{\beta}(y_0)+h(y_0)$ touches $h$ from above at $y_0$. Then by definition of viscosity sub-solutions and by Lemma \ref{lem: vis pot 0} we get 
$$
\left[\omega + (1+\delta)dd^c v_{\beta} \right]^k \wedge \omega^{n-k} (y_0) \geq 0 , \forall k=1\cdots m.
$$
This yields 
$$
(1+\delta)^m H_m(v_{\beta})(y_0) \geq \delta^m\omega^n (y_0),
$$ 
which is a contradiction since in $U$, $H_m(v_{\beta})<\vep^m\omega^n$
and $\delta=\vep (1+\delta)$. 

Therefore, $y_0\notin U$. Since $u\geq f -\vep$ on $X\setminus U$ and since 
$\sup_X |v_{\beta}-u|<\vep$ we get
$$
\phi(y)\leq \phi(y_0)\leq -\delta f(y_0) +2\vep (1+\delta), \ \forall
y\in X.
$$
We thus obtain
$$
h  \leq  (1+\delta) u +\delta \sup_X |f|  +3\vep (1+\delta) .
$$
By letting $\vep\downarrow 0$ (and hence $\delta$ also goes to $0$) we obtain $h\leq u$.

\end{proof}

\begin{cor}
\label{cor: smooth approximation}
For any $\f \in \SHXo$ there exists a sequence 
$(\f_j)\subset \SHXo\cap \Cc^{\infty}(X)$ decreasing to $\f$ on $X$. In particular $\f\in \PmXo$ and the two classes coincide. 
\end{cor}

\begin{proof}
Let $\f$ be a $(\omega,m)$-subharmonic function on $X$. Since $\f$ is in particular upper semicontinuous we can find a sequence $(f_j)$ of smooth functions on $X$ decreasing to $\f$. Note that $(f_j)$ is a priori not $(\omega,m)$-sh on $X$. Let $P(f_j)$ be the upper envelope
of $(\omega,m)$-sh functions lying below $f_j$. It follows from Theorem \ref{thm: upper bound} that $P(f_j)$ is a continuous 
$(\omega,m)$-sh function on $X$ which belongs to $\PmXo$ and satisfies
$$
H_m(P(f_j)) \leq  \ind_{\{P(f_j) = f_j\}} H_m(f_j) .
$$
On the above  the right-hand side has bounded density. Thus it follows from \cite{DK11} (see also \cite{Lu13a}) that for each $j$, $P(f_j)$ is the uniform limit of a sequence of smooth $(\omega,m)$-sh functions.
Therefore, for each $j$ we can find $\f_j\in \SHXo \cap \Cc^{\infty}(X)$ such that 
$$
P(f_j) + \frac{1}{j+1} \leq \f_j \leq P(f_j) + \frac{1}{j}.
$$
Then it is clear that $\f_j$ decreases to $\f$. Now, it follows from 
\cite[Proposition 3.2]{Lu13a} that $\f$ belongs to $\PmXo$ and hence
$\SHXo =\PmXo$. 
\end{proof}

We immediately get the following:
\begin{cor}
\label{cor: continuous proj}
If $h\in \Cc(X)$ then $P(h)$ is a continuous $(\omega,m)$-sh function. Its Hessian measure $H_m(P(h))$ vanishes on $\{P(h)<h\}$.
\end{cor}
\begin{proof}
To prove the first statement it suffices to approximate  $h$ by smooth functions and apply Lemma \ref{lem: proj stable}. To prove the second statement
we first assume that $h$ is smooth on $X$. It follows from Theorem \ref{thm: upper bound} that $P(h)=\lim_{\beta \to +\infty} u_{\beta}$ is the uniform limit of continuous 
$(\omega,m)$-sh functions on $X$. For each $\vep>0$, $H_m(u_{\beta})$
converges uniformly to $0$ on $\{P(h) < h-\vep\}$. This coupled with convergence results in \cite{Lu13a} explain why 
$H_m(P(h))$ vanishes on $\{P(h)<h\}$. The general case follows by approximating $h$ uniformly by smooth functions.
\end{proof}

\section{The Hessian  $m$-capacity}

In Section \ref{sect: Hessian operator} we define Hessian $m$-capacity of any open subset. Now, we know that the Hessian operator is well-defined 
for any bounded $(\omega,m)$-sh function. It turns out that in the definition of the Hessian $m$-capacity one can take the supremum over all 
$(\omega,m)$-sh functions whose values vary from $-1$ to $0$ and not only $\Cc^2$ functions. We then extend this definition to any Borel subset.

For each Borel subset $E\subset X$ we define the Hessian $m$-capacity of $E$ by
$$
\Capm(E) := \sup \left\{\int_E H_m(u) \setdef u\in \SHXo , \ -1\leq u\leq 0 \right\}. 
$$
The following properties of $\Capm$ follow directly from the definition:
\begin{prop}
(i) If $E_1\subset E_2\subset X$ then $\Capm(E_1)\leq \Capm(E_2)$ .

(ii) If $E_1, E_2, \cdots$ are Borel subsets of $X$ then 
$$
\Capm\left(\bigcup_{j=1}^{\infty} E_j\right)\leq \sum_{j=1}^{+\infty} \Capm (E_j) .
$$

(iii) If $E_1\subset E_2\subset \cdots$ are Borel subsets of $X$ then
$$
\Capm \left(\bigcup_{j=1}^{\infty} E_j\right) = \lim_{j\to+\infty} \Capm (E_j) .
$$

\end{prop}

For each Borel subset $E$ set 
$$
h_{m,E} := \sup\{u\in \SHXo \setdef u\leq -1 \ {\rm on }\ E\ {\rm and}\ u\leq 0 \ {\rm on}\ X \}. 
$$
Let $h^{*}_{m,E}$ be the upper semicontinuous regularization of $h_{m,E}$. We call it the relative $m$-extremal function of $E$. 
\begin{thm}

\label{thm: vanishing}
Let $E$ be any Borel subset of $X$ and denote by $h_E$ the relative $m$-extremal function of $E$. Then $h_E$ is a bounded $(\omega,m)$-subharmonic function. Its Hessian measure vanishes on  the open subset $\{h_E<0\} \setminus \bar{E}$.
\end{thm}

\begin{proof}
The first statement is obvious. The second statement follows from the standard balayage argument since it follows from \cite{Pl13} that we can locally solve the Dirichlet problem on any small ball. 

\end{proof}

The following convergence result can be proved in the same way as in \cite{Kol05}.

\begin{lem}

\label{lem: increase}
Let $(\f_j^k)_{j=1}^{+\infty}$ be a uniformly bounded sequence of $(\omega,m)$-sh functions for $k=1,...,m$, which increases almost everywhere to $\f^1,...,\f^m\in \SHXo$ respectively. Then 
$$
u(\omega+dd^c\f_j^1)\wedge (\omega+dd^c\f_j^2) \wedge ...\wedge (\omega+dd^c \f_j^m)\wedge \omega^{n-m}
$$ 
converges weakly in the sense of  Radon measures to 
$$
u(\omega+dd^c\f^1)\wedge (\omega+dd^c\f^2) \wedge ...\wedge (\omega+dd^c \f^m)\wedge \omega^{n-m},
$$ 
for every quasi-continuous function $u$. 
\end{lem}

\begin{lem}

\label{lem: negligible}
Let $\Uc$ be a family of $(\omega,m)$-sh functions which are uniformly bounded from above. Define 
$$
\f(z):= \sup \{u(z) \setdef u\in \Uc \}.
$$
Then the Borel subset $\{\f <\f^*\}$ has zero Hessian $m$-capacity.
\end{lem}
\begin{proof}
By Choquet's lemma we can find a sequence $(\f_j)\subset \Uc$ which increases to $\f$. 

\medbreak

\noindent{\bf Step 1}: Assume that $u\in \SHXo$ and $-1\leq u\leq 0$. We prove by induction on $m$ that 
\begin{equation}
\label{eq: step 1 induction}
\lim_{j\to+\infty}\int_X \f_j  H_m(u) =  \int_X \f^*  H_m(u).
\end{equation}
The result holds for $m=0$ since $\f=\f^*$ almost everywhere (with respect to the Lebesgue measure).  Assume that it holds for $k=1,...,m-1$.  By setting $T:=(\omega +dd^c u)^{m-1}\wedge \omega^{n-m}$ and integrating by parts we get
\begin{eqnarray*}
\int_X \f_j H_m(u) &=& \int_X \f_j H_{m-1}(u) + \int_X \f_j dd^c u \wedge T\\
& = & \int_X \f_j H_{m-1}(u) + \int_X u dd^c \f_j \wedge T\\
& = & \int_X \f_j H_{m-1}(u) + \int_X u (\omega+dd^c \f_j) \wedge T
- \int_X u H_{m-1}(u).\\
\end{eqnarray*}
Since $u$ is quasi-continuous on $X$, by letting $j\to +\infty$ and using Lemma \ref{lem: increase} and using  the induction hypothesis 
we obtain
\begin{eqnarray*}
\lim_{j\to +\infty} \int_X \f_j H_m(u) &=& 
\int_X \f^* H_{m-1}(u) + \int_X u (\omega+dd^c \f^*) \wedge T
- \int_X u H_{m-1}(u)\\
&=&\int_X \f^* H_m(u).
\end{eqnarray*}

\medbreak

\noindent{\bf Step 2}: Since $\Capm$ is $\sigma$-subadditive, it suffices to prove that for each pair
$(\alpha,\beta)$ of rational numbers such that $\alpha<\beta$ the compact subset
$$
K_{\alpha,\beta} := \{x\in X \setdef \f(x) \leq \alpha <\beta \leq \f^*\}
$$
has zero Hessian $m$-capacity. This is an easy consequence of Step 1. The proof is thus complete.

\end{proof}

The outer Hessian $m$-capacity of a Borel subset $E$ is defined by
$$
\Capm^*(E) := \inf \left\{\Capm(U) \setdef E\subset U\subset X, \ U\ {\rm is\ open\ in }\ X \right\}. 
$$
The following properties of $\Capm^*$ follow directly from the definition:
\begin{prop}
(i) If $E_1\subset E_2\subset X$ then $\Capm^*(E_1)\leq \Capm^*(E_2)$ .

(ii) If $E_1, E_2, \cdots$ are Borel subsets of $X$ then 
$$
\Capm^*\left(\bigcup_{j=1}^{\infty} E_j\right)\leq \sum_{j=1}^{+\infty} \Capm^* (E_j) .
$$
\end{prop}

Now we give a formula for the outer Hessian $m$-capacity of any Borel subset of $X$ in terms of its relative $m$-extremal function.

\begin{thm}

\label{thm: cap formula}
Let $E\subset X$ be a Borel subset and $h$ denote the relative 
$m$-extremal function of $E$. Then the outer Hessian $m$-capacity of $E$ is given by
$$
\Capm^*(E) = \int_X (-h)H_m(h) .
$$
The Hessian $m$-capacity satisfies  the following continuity properties:

(i) If $(E_j)_{j\geq 0}$ is an increasing sequence of arbitrary subsets of $X$ and $E := \cup_{j\geq 0} E_j$ then 
$$
\Capm^*(E) = \lim_{j\to+\infty} \Capm^*(E_j).
$$

(ii) If $(K_j)_{j\geq 0}$ is a decreasing sequence of compact subsets of $X$ and 
$K:= \cap_{j\geq 0} K_j$ 
$$
\lim_{j\to+\infty}\Capm^*(K_j)  = \Capm^*(K)= \Capm(K).
$$ 
In particular, $\Capm^*$ is an outer regular Choquet capacity and we have
$$
\Capm^*(B) =\Capm(B)
$$ 
for every Borel set $B$.
\end{thm}
\begin{proof}
It follows from Lemma \ref{lem: negligible} that the subset $\{h>-1\}\cap E$ has zero Hessian $m$-capacity. Thus  we can copy from lines to lines the proof of Theorem 4.2 in \cite{GZ05}. For the last statement 
let us briefly recall the arguments in \cite{BBGZ13}. Since $\Capm^*$
is an (outer regular) Choquet capacity it then follows from Choquet's capacitability theorem that $\Capm^*$ is also inner regular on Borel sets. We thus get
$$
\Capm(B) \leq \Capm^*(B) = \sup_{K\subset B}\Capm^*(K) = \sup_{K\subset B}\Capm (K) = \Capm(B).
$$
\end{proof}

\begin{defi}
Let $E$ be a Borel subset of $X$. The global $(\omega,m)$-subharmonic extremal function of $E$ is $V^*_{m,E}$, where 
$$
V_{m,E} :=\sup \left\{ u\in \SHXo \setdef u\leq 0 \ {\rm on}\ E \right\}.
$$
\end{defi}

\begin{defi}
A subset $E\subset X$ is called $m$-polar if $\Capm^*(E)=0$.
\end{defi}

\begin{lem}
If $E\subset \{\f=-\infty\}$ for some $\f\in \SHXo$ then $E$ is $m$-polar.
\end{lem}
\begin{proof}
It is easy to see that the relative $m$-extremal function of $E$ is $0$. Then the result is obtained by invoking Theorem \ref{thm: cap formula}.
\end{proof}

We now prove the Josefson theorem for $(\omega,m)$-sh functions which generalize \cite[Theorem 7.2]{GZ05}. Let us stress that our proof is more direct without using the local capacity which has  not been available yet. Recall that a local $m$-capacity has been studied in 
\cite{SA12} and \cite{Lu13c} where the metric is flat. For a general 
K\"ahler metric we believe that similar study can be done.
\begin{thm}
If $\Capm^*(E)=0$ then $E\subset \{\f=-\infty\}$ for some $\f\in \SHXo$.
\end{thm}
\begin{proof}
Without loss of generality we can assume that $\bar{E}\neq X$. 
Observe that $\Capa_{\theta,m}^*(E)=0$ for any K\"ahler form $\theta$ since we have $C^{-1}\omega \leq \theta \leq C \omega$ for some positive constant $C$. 
Let 
$V:=V^*_{m,E}$ be the global $m$-extremal function of $E$. 

We prove that $V\equiv +\infty$. Assume by contradiction that it is not the case. Then $V$ is a bounded $(\omega,m)$-sh function on $X$. 
Using a balayage argument as in the proof of Theorem \ref{thm: vanishing} we can prove that $H_m(V)$ vanishes on 
$X\setminus \bar{E}$. Thus $V$ can not be  constant.

Let $M=\sup_X V <+\infty$. We claim that $\p:=(V-M)/M$ is the relative 
$(\theta,m)$-extremal function of $E$ with $\theta=\omega/M$. Indeed, let $u$ be any non-positive $(\theta,m)$-sh function on $X$ such that $u\leq -1$ on $E$. Then $M(u+1)$ is a $(\omega,m)$-sh function on $X$ 
which is non-positive on $E$. By definition of the global $(\omega,m)$-sh extremal function we deduce that $M(u+1)\leq V$, which implies what we have claimed.

Now, since $\Capa_{\theta,m}^*(E)=0$ it follows from Theorem \ref{thm: cap formula} that 
$$
\int_{\{\p<0\}} H_m(\p) =0.
$$
This coupled with the comparison principle reveal that $\p=0$ which implies that $V\equiv M$. The latter is a contradiction since $H_m(V)$
vanishes on the open non-empty set $X\setminus \bar{E}$.

Therefore, $V_{m,E}$ is not bounded from above. We then can find a sequence $(\f_j)\subset \SHXo$ such that $\f_j\equiv 0$ on $E$ and
$\sup_X \f_j\geq 2^{j}$. Consider
$$
\f:= \sum_{j=1}^{+\infty} 2^{-j}(\f_j-\sup_X \f_j).
$$
Then since $\f_j=0$ on $E$ it is easy to see that $\f=-\infty$ on $E$.
It follows from Hartogs' lemma (see \cite{Lu13a}) that 
$$
\int_X (u-\sup_X u)\ \omega^n \geq -C, \ \ \forall u\in \SHXo ,
$$
for some positive constant $C$. It follows that $\f$ is not identically $-\infty$. Hence $\f\in \SHXo$ satisfies our requirement. 
\end{proof}

\section{Energy classes}

For convenience we rescale $\omega$ so that $\int_X \omega^n=1$. It follows from Corollary \ref{cor: smooth approximation} that $\SHXo=\PmXo$. Therefore the complex Hessian operator is well-defined for any bounded $(\omega,m)$-sh function. We will follow \cite{GZ07} to extend the definition of $H_m$ to unbounded $(\omega,m)$-sh functions. Almost all results about the weighted energy classes $\Ec_{\chi}(X,\omega)$ can be extended without effort to the corresponding classes of $(\omega,m)$-subharmonic functions. For this reason we only state the result without proof.

Let $\f\in \SHXo$ and denote by $\f_j$ the canonical approximation of 
$\f$ by bounded functions, i.e. $\f_j:= \max(\f,-j)$. It follows from the comparison principle (see \cite{Lu13a}) that
$$
\ind_{\{\f>-j\}}H_m(\f_j)
$$  
is a non-decreasing sequence of positive Borel measures on $X$. We define $H_m(\f)$ to be its limit. Note that the total mass of 
$H_m(\f)$ varies from $0$ to $1$ and it does not charge $m$-polar sets. 

\begin{defi}
We let $\EcXo$ denote the class of $(\omega,m)$-sh functions having full Hessian mass, i.e.
$$
\EcXo := \left\{ u\in \SHXo \setdef \int_X H_m(u) =1\right\}.
$$
\end{defi}

\begin{lem}
A function $u\in \SHXo$ belongs to $\EcXo$ if and only if
$$
\lim_{j\to+\infty} \int_{\{u\leq -j\}} H_m(\max(u,-j)) = 0 .
$$
The sequence $H_m(\max(\f,-j))$ converges  to $H_m(\f)$
in the sense of Borel measures, i.e. for any Borel subset 
$E\subset X$,
$$
\lim_{j\to +\infty}\int_E H_m(\max(\f,-j))  = \int_E H_m(\f).
$$
\end{lem}

\begin{defi}
Let $\chi$ be an increasing function $\R^-\to \R^-$ such that
$\chi(0)=0$ and $\chi(-\infty)=-\infty$. We let $\Ec_{\chi}(X,\omega,m)$ denote the class of functions $\f$ in $\EcXo$ such that
$\chi\circ \f$ is integrable with respect to $H_m(\f)$. When $\chi(t)=-(-t)^p$, $p>0$ we use the notation $\Ec_m^p(X,\omega)$ 
$$
\Ec^p(X,\omega,m) := \left\{ u\in \EcXo \setdef \int_X |u|^p H_m(u) <+\infty \right\}.
$$
\end{defi}

\begin{lem}
Let $\f\in \EcXo$  and $h:\R^+\rightarrow \R^+$ be a continuous increasing function such that $h(+\infty)=+\infty$. Then 
$$
\int_X h\circ |\f| H_m(\f) <+\infty \Longleftrightarrow \sup_{j\geq 0} 
\int_X h\circ |\f_j| H_m(\f_j) <+\infty ,
$$
where $\f_j:=\max(\f,-j)$. 
\end{lem}

\begin{lem}
If $\f\in \EcXo$ and $\f\leq 0$ there exists a convex increasing function $\chi: \R^-\rightarrow \R^-$ such that $\chi(-\infty)=0$
and $\f\in \Ec_{\chi}(X,\omega,m)$.
\end{lem}

\begin{thm}
Let $\f\in \SHXo$ be such that $\sup_X \f\leq -1$. Let $\chi: \R^-\rightarrow \R^-$ be a smooth convex increasing function such that
$\chi'(-1) \leq 1$ and $\chi'(-\infty) =0$. Then $\chi\circ \f\in \EcXo$.
\end{thm}

The maximum principle and the comparison principle hold for $\EcXo$:
\begin{thm}
Let $\f,\p$ be two functions in $\EcXo$. Then
$$
\ind_{\{\f<\p\}}H_m(\max(\f,\p))= \ind_{\{\f<\p\}}H_m(\p) 
$$
and 
$$
\int_{\{\f<\p\}} H_m(\p) \leq \int_{\{\f<\p\}} H_m(\f).
$$
\end{thm}

\begin{prop}\label{prop: maximum comparison}
Assume that $\f,\p \in \EcXo$ such that $H_m(\f) \geq \mu$
and $H_m(\p) \geq \mu$ for some positive Borel measure $\mu$
on $X$. Then 
$$
H_m(\max(\f,\p)) \geq \mu .
$$
\end{prop}

\begin{thm}

\label{thm: convergence E}
Let $(\f_j)$ be a monotone sequence of functions in $\EcXo$ converging to $\f\in \EcXo$. Then $H_m(\f_j)$ converges weakly to 
$H_m(\f)$.
\end{thm}

\begin{prop}
The set $\EcXo$ is convex. It is stable under the max operation: if $\f,\p \in \SHXo$ are such that $\f\leq \p$ and $\f\in \EcXo$, then 
$\p\in \EcXo$.
\end{prop}
When $m=n$, the class $\Ec(X,\omega,n)$ is exactly $\Ec(X,\omega)$, the class of $\omega$-psh functions having full Monge-Amp\`ere mass, introduced and studied in \cite{GZ07}. 

One can follow the lines in \cite{Dinew09} to prove the "partial comparison principle":
\begin{lem}
Let $T$ be a positive current of type 
$$
T= (\omega+dd^c \phi_1)\wedge \cdots \wedge (\omega +dd^c \phi_{k})\wedge \omega^{n-m},\ \ k<m,
$$
where the $\phi_j$'s are functions in $\EcXo$. Let $u,v\in \EcXo$. Then
$$
\int_{\{u<v\}} (\omega +dd^c v)^{m-k} \wedge T \leq \int_{\{u<v\}} (\omega +dd^c u)^{m-k} \wedge T .
$$
\end{lem}

\begin{thm}
\label{thm: EcXn subset EcXm}
$\Ec(X,\omega,n) \subset \Ec(X,\omega,n-1)\subset \cdots \subset  \Ec(X,\omega,1)$.
\end{thm}
\begin{proof}
Fix $p<m$ and $\f \in \Ec(X,\omega,m)$. Let $\f_j:=\max(\f,-j)$ be the canonical approximation sequence of $\f$. We are to prove that 
$$
\int_{\{\f>-j\}} H_{m-1}(\f_j) \longrightarrow 1.
$$
From the partial comparison principle above we get 
$$
\int_{\{\f_j>-j\}} (\omega +dd^c \f_j)^{p}\wedge \omega^{m-p} \wedge \omega^{n-m} \geq \int_{\{\f_j>-j\}} (\omega +dd^c \f_j)^{p}\wedge (\omega+dd^c \f_j)^{m-p} \wedge \omega^{n-m}.
$$
From this and since $\f\in \Ec(X,\omega,m)$ we get the conclusion.
\end{proof}

\begin{exa}
Let $z$ be a local coordinate of $X$ and consider 
$$
\f := \vep \theta \log |z| , 
$$
where $\theta$ is a cut-off function and $\vep>0$ is a very small constant  so that $\f\in \SHXo$. Then $\f\in \EcXo$ for any $m<n$
but $\f\notin \Ec(X,\omega,n)$.
\end{exa}

\section{The variational method}
The variational method has first introduced in 
\cite{BBGZ13} to solve degenerate complex Monge-Amp\`ere equations 
on compact K\"ahler manifolds. A local version of this approach has been developed in \cite{ACC10}. 

Due to some similar structure one expects that this method can also be applied for the complex Hessian equation. In the local setting with a standard K\"ahler metric  the first-named author   \cite{Lu13c} has used  this method to solve  degenerate complex Hessian equations in 
$m$-hyperconvex domains of $\C^n$. To make it 
available for the compact setting the principal steps are: first to smoothly regularize
singular $(\omega,m)$-sh functions and  then to prove an othorgonal relation. Both of them have been proved in Section 3. In the sequel we briefly recall the techniques of \cite{BBGZ13}. Most of the proof will be omitted due to similarity and repetition.
\subsection{The energy functional}
\begin{defi}
Let $\f$ be a bounded $(\omega,m)$-sh function on $X$. We define
$$
E(\f) := \frac{1}{m+1} \sum_{k=0}^{m} \int_X \f (\omega +dd^c\f)^k \wedge \omega^{n-k}
$$
to be the energy of $\f$.
For any  $u\in \SHXo$ the energy of $u$ is defined by 
$$
E(u) := \inf \left\{ E(\f) \setdef \f \in \SHXo\cap L^{\infty}(X) , \ u\leq \f \right\}.
$$
\end{defi}

\begin{lem}
For any $\f \in \Ect$ such that $\f\leq 0$ we have 
$$
\int_X \f H_m(\f) \leq E(\f) \leq \frac{1}{m+1} \int_X \f H_m(\f).
$$
The class $\Ect$ consists of finite energy $(\omega,m)$-subharmonic functions. If $(\f_j)$ is a sequence in $\Ect$ decreasing to $\f$ such that 
$$
\inf_j E(\f_j) >-\infty
$$
then $\f\in \Ect$ and $E(\f) = \lim_{j \to +\infty} E(\f_j)$
\end{lem}

\begin{lem}\label{lem: primitive}
The functional $E$ is a primitive of the complex Hessian operator. More precisely, whenever $\f+tv$ belongs to $\Ect$ for small $t$,
$$
\frac{dE(\f+tv)}{dt}\vert_{t=0} = \int_X v H_m(\f).
$$
The functional $E$ is concave increasing, satisfies $E(\f + c) = E(\f) + c$ for
all $c \in \R, \f\in \Ect$ , and the cocycle condition
$$
E(\f) - E(\p) = \frac{1}{m+1} \sum_{j=0}^m \int_X (\f-\p) (\omega
+dd^c \f)^j \wedge (\omega+dd^c \p)^{m-j}\wedge \omega^{n-m},
$$
for all $\f,\p \in \Ect$. Moreover,
$$
\int_X (\f-\p) H_m(\f) \leq E(\f) - E(\p) \leq \int_X (\f-\p) H_m(\p).
$$
\end{lem}

\begin{proof}
The proof is a trivial adaptation of \cite{BBGZ13}. 
\end{proof}

\begin{lem}
The functional $E$ is upper semicontinuous with respect to the $L^1$
topology on $\SHXo$.
\end{lem}
\begin{proof}
Assume that $(\f_j)$ is a sequence in $\SHXo$ converging to $\f\in \SHXo$ in $L^1$. We are to prove that 
$$
\limsup_{j\to+\infty} E(\f_j) \leq E(\f).
$$
If the limsup is $-\infty$ there is nothing to do. Thus we can assume that $E(\f_j)$ is uniformly bounded from below. Then since
$$
E(\f_j) \leq  \int_X \f_j \omega^n
$$
the sequence $(\f_j)$ stays in a compact subsets of $\SHXo$.  Assume that $\f_j\rightarrow \f \in \SHXo$ in $L^1(X)$. Set
$$
\p_j:=(\sup_{k\geq j} \f_k)^*.
$$
Then $\p_j$ decreases to $\f$. Since $E$ is increasing we get a uniform lower bound for $E(\p_j)$. Thus $\f$ belongs to $\EcXo$ and
$$
E(\f) = \lim_{j\to+ \infty} E(\p_j) \geq \limsup_{j\to+\infty} E(\f_j).
$$
\end{proof}

\begin{lem}
For each $C>0$ the set
$$
\EctC := \{\f\in \Ect\setdef \sup_X \f \leq 0 , \ E(\f) \geq -C\}
$$
is a compact convex subset of $\SHXo$.
\end{lem}
\begin{proof}
The convexity of $\EctC$ follows from the concavity of $E$. The compactness follows from the upper semicontinuity of $E$.  
\end{proof}

The following volume-capacity estimate is due to Dinew and Ko{\l}odziej \cite{DK11}:
\begin{lem}
\label{lem: vol-cap DK}
Let $1<p<\frac{n}{n-m}.$ There exists a constant $C=C(p,\omega)$ such that
for every Borel subset $K$ of $X$, we have
$$
V(K)\leq C\cdot \Capm(K)^p,
$$
where $V(K):=\int_K\omega^n$.
\end{lem}

\begin{cor}\label{cor: Blocki}
Let $\f\in \SHXo$. Then $\f\in L^p(X,\omega^n)$ for any $p<\frac{n}{n-m}$. 
\end{cor}
\begin{proof}
We can assume that $\sup_X \f=1$. Fix $p<n/(n-m)$ and $q$ such that $p<q<n/(n-m)$. It follows from \cite[Corollary 3.19]{Lu13a} and the previous volume-capacity estimate that 
\begin{eqnarray*}
\int_X (-\f)^p \omega^n & =& 1+ p\int_1^{+\infty} t^{p-1}V(\f<-t) dt\\
&\leq & 1 + C_q p\int_1^{+\infty} t^{p-1}\left[\Capm(\f<-t)\right]^q dt\\
&\leq &  1 + C_q C p \int_1^{+\infty} t^{p-q-1} dt <+\infty.
\end{eqnarray*}
\end{proof}
One expects that Corollary \ref{cor: Blocki} holds for any $p<\frac{nm}{n-m}$. In the local context where $\omega$ is the standard K\"ahler metric, it was known as B{\l}ocki's conjecture.

\begin{lem}\label{lem: cap and E1}
Fix $\f\in \SHXo$. If 
$$
\int_{0}^{+\infty} t^m \Capm(\f<-t) dt <+\infty 
$$
then $\f\in \Ect$. Conversely  for each $C>0$,
$$
\sup \left\{ \int_0^{+\infty} t\Capm(\f<-t) dt \setdef \f\in \Ec_{m,C}^1(X,\omega) \right\} <+\infty .
$$
\end{lem}
\begin{proof}
Fix $\f\in \SHXo$. We can assume that $\sup_X \f=-1$. Observe that for 
$t\geq 1$, the function $1 + t^{-1} \max(\f, −t)$ is $(\omega,m)$-sh with values in $[0, 1]$, hence
$$
H_m(\max(\f,-t)) \leq t^m \Capm .
$$
Let us prove the first assertion. If 
$$
\int_{0}^{+\infty} t^m \Capm(\f<-t) dt <+\infty 
$$
then in particular $t^m \Capm(\f<-t)$ converges to $0$ as $t\to +\infty$. This coupled with the above observation yields
$$
\int_{\{\f\leq -t\}}H_m(\max(\f,-t)) \longrightarrow 0 ,
$$
which implies that $\f\in \EcXo$. Now by the comparison principle 
$H_m(\max(\f,-t))$ coincides with $H_m(\f)$ on the Borel set $\{\f>-t\}$. We thus get
\begin{eqnarray*}
\int_X (-\f)H_m(\f) &=& 1+ \int_{1}^{+\infty} H_m(\f)(\f\leq -t) dt\\
&=& 1+ \int_{1}^{+\infty} \left[1-H_m(\f)(\f> -t)\right] dt\\
&=& 1+ \int_{1}^{+\infty} \left[1-H_m(\max(\f,-t))(\f> -t)\right] dt\\
&\leq & 1+ \int_{1}^{+\infty} H_m(\max(\f,-t))(\f \leq -t)) dt\\
&\leq & 1+ \int_{1}^{+\infty} \Capm(\f \leq -t) dt \\
&< & +\infty,
\end{eqnarray*}
which yields $\f\in \EcXo$.

We now prove the second assertion. The proof is slightly different from the classical Monge-Amp\`ere equation due to a lack of integrability (it is not very clear that $\int_X \f^2 \omega^n<+\infty$). Fix $u\in \SHXo$ with values in 
$[-1,0]$. Observe that 
$$
(\f<-2t) \subset (t^{-1}\f< u-1) \subset (\f<-t) .
$$
It follows from the comparison principle that 
$$
\int_{\{\f<-2t\}} H_m(u) \leq \int_{\{\f<-t\}}H_m(t^{-1}\f) .
$$
Expanding $H_m(t^{-1}\f)\leq (t^{-1}(\omega+dd^c \f) + \omega)^m \wedge \omega^{n-m}$ yields
\begin{multline*}
\int_2^{+\infty}  t\Capm(\f<-t) = 4\int_1^{+\infty} t\Capm(\f<-2t)dt 
\\
\leq  4\int_1^{+\infty}  t \vol(\f<-t) dt + 4 \sum_{j=1}^m \binom{m}{j}\int_X (-\f) \omega_{\f}^j\wedge\omega^{n-j}.
\end{multline*}
The last term is finite and uniformly bounded in $\f\in \EctC$. Fix $1<p<\frac{n}{n-m}$ 
and $0<\gamma<1$. Using H\"older inequality we get
\begin{multline*}
\int_1^{+\infty}  t \vol(\f<-t) dt =  \int_1^{+\infty}  t \vol(\f<-t)^{\gamma} \vol(\f<-t)^{1-\gamma} dt \\
\leq  \left[\int_1^{+\infty}  t \vol(\f<-t)^{q\gamma} dt\right]^{1/q} \left[\int_1^{+\infty}  t \vol(\f<-t)^{r(1-\gamma)} dt\right]^{1/r}\\
\leq   A\left[\int_1^{+\infty}  t \Capm(\f<-t)^{pq\gamma} dt\right]^{1/q} \left[\int_1^{+\infty}  t \Capm(\f<-t)^{pr(1-\gamma)} dt\right]^{1/r}\\
\leq  A \left[\int_1^{+\infty}  t \Capm(\f<-t) dt\right]^{1/q} \left[\int_1^{+\infty}  t^{1-pr(1-\gamma)} dt\right]^{1/r}.
\end{multline*}
Here, $1/q +1/r =1$ and we have chosen $\gamma$ so that $pq\gamma=1$ and $pr(1-\gamma) >2$. Such a choice is always possible. The constant 
$A$ is also uniform in $\f\in \EctC$ since $\sup_X \f \geq E(\f)\geq -C$ and 
$$
\Capm(u<-t) \leq C/t
$$
for a uniform constant $C$ as follows from \cite{Lu13a}.

By considering $\f_j:=\max(\f,-j)$ and applying what we have done so far we get
$$
C_j \leq A \cdot C_j^{1/q} + B,
$$
where $C_j:= \int_1^{+\infty} t\Capm(\f_j<-t)$ and $A,B$ are universal constant. Letting $j\to+\infty$ we get the result.
\end{proof}

\subsection{Upper semicontinuity }
Let $\mu$ be a probability measure on $X$. The functional $\Fc_{\mu}$
is defined by
$$
\Fcm (\f) := E(\f) -\int_X \f d\mu .
$$
\begin{lem}\label{lem: Banach-Sak}
Let $\mu$ be a probability measure which does not charge $m$-polar sets. Let $(u_j)\subset \SHXo$ be a sequence which converges in
$L^1(X)$ towards $u\in \SHXo$. If 
$\sup_{j\geq 0} \int_X u_j^2d\mu <+\infty $
then 
$$
\int_X u_j d\mu \longrightarrow \int_X u d\mu .
$$
\end{lem}
\begin{proof}
Since  $\int_X u_jd\mu$ is bounded it suffices to prove that  every cluster point is $\int_X u d\mu.$  
Without loss of generality we can assume that $\int_X u_j d\mu $ converges.  Since the sequence 
  $u_j$ is bounded in $L^2(\mu)$,  one can apply Banach-Saks theorem  to extract a subsequence
   (still denoted by  $u_j$) such that
$$
\f_N:=\frac{1}{N}\sum_{j=1}^Nu_j
$$
converges in $L^2(\mu)$ and  $\mu$-almost everywhere to $\f.$ Observe also that
 $\f_N\to u$ in $L^1(X)$.  For each   $j\in\N$  set
$$
\p_j:=(\sup_{k\geq j}\f_k)^{*}.
$$
Then $\p_j\downarrow u$ in $X$.  But $\mu$ does not charge the $m$-polar set
$$
\left\{(\sup_{k\geq j}\f_k)^{*}>\sup_{k\geq j}\f_k\right\}.
$$
We thus get
$\p_j=\sup_{k\geq j}\f_k$ $\mu$-almost everywhere. Therefore, 
 $\p_j$ converges to $\f$ $\mu$-almost everywhere hence $u=\f$ $\mu$-almost everywhere. This yields
$$
\lim_j \int_X u_jd\mu = \lim_j \int_X \f_jd\mu =\int_X u d\mu.
$$
\end{proof}

\begin{lem}\label{lem: u.s.c.}
Let $\mu$ be a probability measure on $X$ such that 
$$
\mu (K) \leq A \Capm(K), \ \forall K\subset X,
$$
for some positive constant $A$. Then the functional $\Fc_{\mu}$ is upper semicontinuous on each compact subset $\EctC$, $C>0$.
\end{lem}
\begin{proof}
Let $(\f_j)$ be a sequence in $\EctC$ converging in $L^1(X)$ to
$\f\in \EctC$. We can assume that $\f_j\leq 0$. It follows from Lemma \ref{lem: cap and E1} that 
$$
\int_X (-\f_j)^2 d\mu \leq 2 \int_0^{+\infty}t\mu(\f_j< -t) dt
\leq 2A \int_0^{+\infty} t\Capm(\f_j<-t) dt \leq 2AC',
$$ 
for a positive constant $C'$. From Lemma \ref{lem: Banach-Sak} we thus get
$$
\int_X \f_j d\mu \longrightarrow \int_X \f d\mu .
$$ 
This coupled with the upper semicontinuity of $E$ yield the result.
\end{proof}

\begin{defi}
We say that the functional $\Fcm$ is proper if whenever 
$\f_j\in \Ect$ are such that $E(\f_j)\rightarrow -\infty$  and
$\int_X \f_j=0$ then $\Fcm(\f_j) \to -\infty$.
\end{defi}
\begin{lem}\label{lem: properness}
Let $\mu$ be a probability measure on $X$ such that $\EctC\subset L^1(\mu)$. The functional $\Fcm$ is proper: there exists $C > 0$ such that for all $\f\in \Ect$ with $\int_X \f \omega^n=0$ we have
$$
\Fcm(\f) \leq E(\f) + C|E(\f)|^{1/2}. 
$$
\end{lem}
\begin{proof}
Arguing by contradiction we can prove that
$$
\sup \left\{\int_X(-\p) d\mu \setdef \p\in \EctC \right\} <+\infty, \forall C>0.
$$
Now we can repeat the arguments in \cite{BBGZ13}.
\end{proof}

\subsection{The projection theorem}
Let $f$ be an upper semicontinuous function on $X$. Recall that
the projection of $f$ on $\SHXo$ is defined by 
$$
P(f) := \sup\left\{ u\in \SHXo \setdef u\leq f \right\}.
$$

\begin{lem}\label{lem: BB08}
Let $u,v$ be continuous function on $X$. Then
$$
E\circ P(u+v)-E\circ P(u) = \int_0^1 \left[ \int_X v H_m(P(u+tv)) \right] dt.
$$
\end{lem}
\begin{proof}
One could prove the lemma by following  \cite{BB08}. But we give here a slightly different (and simpler) proof using the same ideas.
%By using dominated convergence theorem and the continuity of the Hessian operator $H_m$ we can assume that $v$ is smooth on $X$. By rescaling $v$ we also can assume that $v\in \SHXo$.

Observe that it is equivalent to showing that 
\begin{equation}
\label{eq: BB08 1}
\frac{d E\circ P(u+tv)}{dt}\big \vert_{t=0} = 
\int_X v H_m(P(u)).
\end{equation}
By changing $v$ to $-v$ it suffices to take care of the right derivative. Fix $t>0$. It follows from Lemma \ref{lem: primitive} that 
\begin{eqnarray*}
\int_X \frac{ P(u+tv)-P(u)}{t} H_m(P(u+tv))& \leq &  \frac{E\circ P(u+tv)-E(P(u))}{t}\\
& \leq &  \int_X \frac{ P(u+tv)-P(u)}{t} H_m(P(u)).
\end{eqnarray*}
Since $\int_X (u-P(u)) H_m(P(u))=0$ as follows from Theorem \ref{theorem: orthogonal} the second inequality above yields
the inequality "$\leq$" in (\ref{eq: BB08 1}). On the other hand 
the first inequality above coupled with the orthogonal relation gives 
\begin{eqnarray*}
\frac{E\circ P(u+tv)-E(P(u))}{t} &\geq &\int_X \frac{ P(u+tv)-P(u)}{t}
H_m(P(u+tv)) \\
& = & \int_X \frac{ u+tv-P(u)}{t}
H_m(P(u+tv)) \\
& \geq  & \int_X v
H_m(P(u+tv)) .
\end{eqnarray*}
By letting $t\to 0^+$ we get the inequality "$\geq$" in (\ref{eq: BB08 1}) since $H_m$ is continuous under uniform convergence. The proof is thus complete.
\end{proof}

\begin{thm}
Fix $\f \in \Ect$ and $v \in \Cc(X,\R)$. Then the function 
$t \mapsto E\circ P(\f+tv)$ is differentiable at zero, with 
$$
\frac{d E\circ P(\f+tv)}{dt}\big \vert_{t=0} = \int_X v H_m(\f).
$$
\end{thm}
\begin{proof}
As in the previous lemma it suffices  to prove that 
$$
E\circ P(\f+v)-E\circ P(\f) = \int_0^1 \left[ \int_X v H_m(P(\f+tv)) \right] dt,
$$
for every $\f \in \Ect$ and $v\in \Cc(X,\R)$. It follows from our approximation theorem (Theorem \ref{theorem: smooth approx}) that we can find a sequence of smooth $(\omega,m)$-sh functions decreasing to 
$\f$. By the continuity of $H_m$ we thus can assume that $\f$ is smooth. The result now follows from Lemma \ref{lem: BB08}.
\end{proof}

\section{Resolution of the degenerate complex Hessian equation}
Let $\mu$ be a probability measure on $X$ which does not charge 
$m$-polar sets. We study the following degenerate complex Hessian equation
\begin{equation}
\label{eq: hes energy}
(\omega +dd^c \f)^m \wedge \omega^{n-m} = \mu .
\end{equation}

\begin{thm}\label{thm: dirichlet principle}
Let $\mu$ be a probability measure such that $\mu\leq A \Capm$ for some positive constant $A$ . If $\Fcm$ is proper, then there exists 
$\f\in \Ect$ which solves (\ref{eq: hes energy}) and such that
$$
\Fcm(\f) = \sup_{\Ect} \Fcm .
$$
\end{thm}
\begin{proof}
The proof is a word-by-word copy of \cite{BBGZ13}. We recall the arguments below.

Since $\Fcm$ is invariant by translations and proper, we can find 
$C > 0$ so large that
$$
\sup_{\Ect} \Fcm  = \sup_{\EctC} \Fcm .
$$
Recall that by definition
$$
\EctC := \{\f\in \Ect\setdef \sup_X \f\leq 0 , \ E(\f) \geq -C\}
$$
is a compact convex subset of $\SHXo$.
It follows from Lemma \ref{lem: u.s.c.} that $\Fcm$ is upper semi-continuous on $\EctC$, thus we can find $\f \in \EctC$ which maximizes the functional $\Fcm$ on $\Ect$.

Fix $v\in \Cc(X,\R)$ an arbitrary continuous function on $X$ and consider
$$
g(t) := E\circ P(\f+tv) - \int_X (\f+tv) d\mu , t\in \R .
$$
Then for every $t\in \R$, 
$$
g(t) \leq E\circ P(\f+tv) - \int_X P(\f+tv) d\mu =\Fcm(P(\f+tv)) \leq \Fcm(\f) =g(0).
$$
Thus $g$ attains its maximum at $0$ and hence by differentiability of $g$ at $0$ we have $g'(0)=0$ which implies
$$
\int_X vd\mu = \int_X v H_m(\f) .
$$
Since $v$ has been chosen arbitrarily the conclusion follows.
\end{proof}

\begin{thm}
Let $\mu$ be a probability measure on $X$. Then $\Ect\subset L^1(\mu)$
if and only if $\mu=H_m(\f)$ for some $\f\in \Ect$.
\end{thm}
\begin{proof}
If $\mu=H_m(\f)$ for some $\f\in \EcXo$ then for any $\p\in \EcXo$,
$$
\int_X \p H_m(\f) >-\infty,
$$
since by the comparison principle  we can prove that (see \cite[Proposition 2.5]{GZ07})
$$
\int_X \p H_m(\f) \geq 2(m+1) (E(\f) +E(\p)) >-\infty.
$$
Assume now that $\Ect\subset L^1(\mu)$. In particular, $\mu$
does not charge $m$-polar sets. Observe first that the set
$$
\Mcc:=\{\nu \in \Pc(X) \setdef \nu \leq  \Capm\},
$$
where $\Pc(X)$ is the space of probability measures on $X$,
is a compact convex subset of $\Pc(X)$. Indeed, the convexity is clear while the compactness follows from the outer regularity of the $m$-capacity (see Theorem \ref{thm: cap formula}). 
Using \cite{Rai69}, we project $\mu$ on this compact convex set (the original idea of this proof is due to Cegrell \cite{Ceg98}) 
$$
\mu = f\nu +\sigma ,
$$
where $\nu \in \Mcc$, $0\leq f \in L^1(\nu)$ and $\sigma \perp \Mcc$.
Since $\mu$ vanishes on $m$-polar sets one has $\sigma\equiv 0$.
Set $\mu_j:=c_j\min(f,j)\nu$ where $c_j$ is a normalization constant so that $\mu_j$ is a probability measure. Since $\mu_j \leq jc_j  \Capm$ it follows from Theorem \ref{thm: dirichlet principle} that there exists $\f_j \in \EcXo$ such that $\mu_j=H_m(\f_j)$. We normalize $\f_j$ so that $\sup_X \f_j=0$. We can also assume that 
$\f_j \rightarrow \f\in \SHXo$ in $L^1(X)$. Now
$$
|E(\f_j)| \leq \int_X (-\f_j)H_m(\f_j) \leq c_j\int_X (-\f_j)d\nu \leq 
C|E(\f_j)|^{1/2},
$$
as follows from Lemma \ref{lem: properness}. It follows that $E(\f_j)$ is uniformly bounded and hence $\f\in \Ect$. Now consider
$$
\phi_j:=(\sup_{k\geq j}\f_k)^*.
$$
Then $\phi_j \downarrow \f$ and it follows from Proposition \ref{prop: maximum comparison} that 
$$
H_m(\phi_j) \geq  \min(f,j) \nu .
$$
Hence $H_m(\f) \geq \mu$ whence equality since both of them are probability measures.
\end{proof}

\begin{thm}
Let $\mu$ be a probability measure on $X$ which does not charge $m$-polar sets. Then there exists  $\f\in \EcXo$  such that $H_m(\f)=\mu$.
\end{thm}
\begin{proof}
One can repeat the arguments in \cite{BBGZ13}. 
\end{proof}

\subsection*{Concluding remarks}

The principal result of this paper is the regularization theorem. It is amazing that we can directly regularize any $(\omega,m)$-sh functions by solving appropriate complex Hessian equations. On the way to regularize singular functions we also proved the orthogonal relation is the second amazing thing. The classical method to prove such a thing is to use the balayage argument which is now possible thanks to the resolution of the corresponding local Dirichlet problem
\cite{Pl13}. 

One can also carry a similar study of a potential theory for $(\omega,m)$-subharmonic functions in $\C^n$ with $\omega$ being 
any K\"ahler metric.

\end{document}